\documentclass[12pt]{amsart}
\usepackage[utf8]{inputenc}
\usepackage[T2A]{fontenc}
\usepackage[english]{babel}
\usepackage{fullpage}
\usepackage{amsmath,amsthm}
\usepackage{amssymb,latexsym}
\usepackage{enumerate}
\usepackage[all,cmtip]{xy}

\newtheorem{thm}{Theorem}[section]

\newtheorem{lem}[thm]{Lemma}

\newtheorem{stmt}{Proposition}

\theoremstyle{definition}

\newtheorem{exa}[thm]{Example}

\newtheorem*{xrem}{Remark}
\newtheorem*{question}{Question}
\numberwithin{equation}{section}

\DeclareMathOperator{\GL}{GL}

\DeclareMathOperator{\SO}{SO}

\DeclareMathOperator{\GroupB}{B}
\DeclareMathOperator{\GroupG}{G}
\DeclareMathOperator{\GroupF}{F}
\DeclareMathOperator{\GroupH}{H}
\DeclareMathOperator{\GroupT}{T}
\DeclareMathOperator{\GroupU}{U}
\DeclareMathOperator{\GroupR}{R}
\DeclareMathOperator{\GroupS}{S}

\DeclareMathOperator{\Mat}{Mat}

\DeclareMathOperator{\Spec}{Spec}
\DeclareMathOperator{\pr}{pr}

\def \isom {\cong}
\def \AA {\mathbb{A}}

\def \PP {\mathbb{P}}

\def \CC {\mathbb{C}}
\def \kk {\CC}

\def \catquot {/\hspace{-3pt}/}
\def \dcosets #1#2#3 {#1 \hspace{-1pt} \backslash\hspace{-3pt}\backslash\hspace{-0.8pt}{#2}\hspace{-1pt}\slash\hspace{-3pt}\slash #3 \hspace{1pt}}
\def \predcosets #1#2#3 {#1 \backslash{#2}\slash #3 \hspace{1pt}}

\begin{document}

\renewcommand{\thefootnote}{}



\renewcommand{\thefootnote}{\arabic{footnote}}
\setcounter{footnote}{0}

\title{On existence of double coset varieties}

\author{Artem Anisimov}
\address{Department of Higher Algebra, Faculty of Mechanics and Mathematics, Lomonosov Moscow State University Leninskie Gory 1, GSP-1, Moscow 119991, Russia}
\email{aanisimov@inbox.ru}

\maketitle

\begin{abstract}
Let~$\GroupG$ be a complex affine algebraic group and~$\GroupH, \GroupF \subset \GroupG$ be closed subgroups. The homogeneous space~$\GroupG / \GroupH$ can be equipped with structure of a smooth quasiprojective variety. The situation is different for double coset varieties $\dcosets{\GroupF}{\GroupG}{\GroupH}$. In this paper we give examples showing that the variety~$\dcosets{\GroupF}{\GroupG}{\GroupH}$ does not necessarily exist. We also address the question of existence of $\dcosets{\GroupF}{\GroupG}{\GroupH}$ in the category of constructible spaces and show that under sufficiently general assumptions $\dcosets{\GroupF}{\GroupG}{\GroupH}$ does exist as a constructible space.
\end{abstract}

\section{Introduction}
Let~$\GroupG$ be a complex affine algebraic group and~$\GroupH \subseteq \GroupG$ be a closed subgroup. By Chevalley Theorem the set of left~$\GroupH$-cosets can be equipped with a uniquely defined structure of a smooth quasiprojective variety such that~$\GroupG$ act morphically on~$\GroupG / \GroupH$. Moreover, the projection~$\GroupG \rightarrow \GroupG / \GroupH$ is a geometric quotient for the action of~$\GroupH$ on~$\GroupG$ by right multiplication.

The construction of the homogeneous space~$\GroupG / \GroupH$ has a natural generalisation: one can take another subgroup~$\GroupF \subset \GroupG$ and consider double cosets, i.~e. the sets~$\GroupF g \GroupH$, $g \in \GroupG$. These cosets are orbits of the action of~$\GroupF \times \GroupH$ on~$\GroupG$ given by the formula~$(f,h) \circ g = fgh^{-1}$. It is clear that such action, unlike the action of~$\GroupH$ on~$\GroupG$ by multiplication, can have orbits of different dimensions, thus it does not necessarily admit a geometric quotient. Because of this we consider a weaker quotient, namely, a categorical one.

\textit{The double coset variety}~$\dcosets{\GroupF}{\GroupG}{\GroupH}$ is defined to be the underlying space of the categorical quotient~$\GroupG \rightarrow \dcosets{\GroupF}{\GroupG}{\GroupH}$ with respect to the described action of~$\GroupF \times \GroupH$, if this quotient exists. If the subgroups~$\GroupF$ and~$\GroupH$ are reductive then this variety exists and coincides with the spectrum~$\Spec ({}^{\GroupF} \kk[\GroupG]^{\GroupH})$ of the algebra of regular functions on~$\GroupG$ invariant under the action of~$\GroupF \times \GroupH$. Moreover, if~$\GroupG$ is also reductive then by a result of Luna~\cite{LunaClosedOrbits} the action $\GroupF \times \GroupH : \GroupG$ is stable\footnote{Reductivity of~$\GroupG$ is essential: consider the group~$\GroupB$ of the upper-triangular matrices and its subgroup~$\GroupT$ of the diagonal matrices; the action~$\GroupT \times \GroupT : \GroupB$ is not stable.}, hence~$\dcosets{\GroupF}{\GroupG}{\GroupH}$ parametrises generic (closed) double cosets.

In this paper we consider the case when the subgroups~$\GroupF$ and~$\GroupH$ are not reductive. In this setting one can not guarantee that~$\dcosets{\GroupF}{\GroupG}{\GroupH} = \Spec ({}^{\GroupF}\kk[\GroupG]^{\GroupH})$; moreover,~$\dcosets{\GroupF}{\GroupG}{\GroupH}$ does not necessarily exist. To illustrate this we give the following examples:

\begin{enumerate}[I.]
\item\label{example0}
A unipotent group~$\GroupG$ and a subgroup~$\GroupU$ of~$\GroupG$ such that the variety~$\dcosets{\GroupU}{\GroupG}{\GroupU}$ does not exist.
\item\label{example1}
A reductive group~$\GroupG$ and two subgroups~$\GroupF, \GroupH$ such that the variety~$\dcosets{\GroupF}{\GroupG}{\GroupH}$ does not exist.
\item\label{example2}
A semisimple group~$\GroupG$ and two subgroups~$\GroupF, \GroupH$ such that the algebra of~$\GroupF \times \GroupH$-invariant regular functions~$R = {}^{\GroupF} \kk[ \GroupG ]^{\GroupH}$ is finitely generated and the natural morphism~$\pi : \GroupG \rightarrow \Spec R$ is surjective, but~$\pi$ is not a categorical quotient.
\end{enumerate}

It is interesting to remark that though~$\dcosets{\GroupU}{\GroupG}{\GroupU}$ considered in Example~\ref{example0} does not exist as an algebraic variety, it does exist as a constructible space. Thus, here we observe the same phenomenon as in~\cite{ArzhCelik}, \cite{CampoNeuen} and~\cite{Celik}, namely, an action that admits no quotient in the category of algebraic varieties does admit one in the category of constructible spaces.

In Example~\ref{example2} the categorical quotient~$\dcosets{\GroupF}{\GroupG}{\GroupH}$ exists in the category of algebraic varieties; its underlying space is the blow-up of~$\Spec ({}^{\GroupF} \kk[ \GroupG ]^{\GroupH})$ at one point; moreover, in this example the categorical quotient separates generic double cosets.

The author would like to thank I.~V.~Arzhantsev for stating the problem and helpful discussions.

\section{Preliminaries on categorical quotients}
Let an algebraic group~$\GroupG$ act on an algebraic variety~$X$. Recall that the categorical quotient of this action is a~$\GroupG$-invariant (i.~e., constant on \hbox{$\GroupG$-orbits}) morphism \hbox{$\pi_{\GroupG} : X \rightarrow Y$} such that every~$\GroupG$-invariant morphism $\varphi : X \rightarrow Z$ factors uniquely through~$\pi_{\GroupG}$, that is, there is a unique morphism~$\tilde{\varphi}$ making the following diagram commutative:
\begin{equation*}
\xymatrix{
X		\ar[rr]^{\varphi}	\ar[rd]_{\pi_{\GroupG}}		&&		Z.		\\
&		Y		\ar[ru]_{\tilde{\varphi}}
}
\end{equation*}
The universal property of~$\pi_{\GroupG}$ implies that~$Y$ is defined uniquely up to isomorphism. Remark that~$\pi_{\GroupG}$ is necessarily surjective. By abuse of language we will sometimes call the variety~$Y = X \catquot \GroupG$ the categorical quotient.

If~$\GroupG$ is reductive and~$X$ is affine then the categorical quotient for the action~$\GroupG : X$ is~$\pi_{\GroupG} : X \rightarrow Y = \Spec \kk[X]^{\GroupG}$ with morphism~$\pi_{\GroupG}$ corresponding to inclusion~$\kk[X]^{\GroupG} \subset \kk[X]$; in this case~$\pi_{\GroupG}$ has an important additional property: it separates closed orbits. If~$\GroupG$ is not reductive then the quotient~$X \catquot \GroupG$ does not necessarily exist. Examples of actions not admitting a categorical quotient are given in~\cite[4.3]{VinbergPopov}, \cite{CampoNeuen}, \cite{ArzhCelik}. Let us point out one example that we will make use of.
\begin{exa}\cite[4.3]{VinbergPopov}\label{simple-example}
There is no categorical quotient for the action of a one-dimensional unipotent group~$\GroupU$ on space~$\Mat_{2 \times 2}$ of~$2 \times 2$-matrices given by the formula
\begin{equation*}
\lambda \circ \begin{pmatrix}a_{11} & a_{12} \\ a_{21} & a_{22}\end{pmatrix} = 
\begin{pmatrix}1 & \lambda \\ 0 & 1\end{pmatrix} \begin{pmatrix}a_{11} & a_{12} \\ a_{21} & a_{22}\end{pmatrix}.
\end{equation*}
Remark that we have~$\kk [\Mat_{2 \times 2}]^{\GroupU} = \kk [ a_{21}, a_{22}, \det ]$ and the canonical morphism $\pi : \Mat_{2 \times 2} \rightarrow \Spec \kk [\Mat_{2 \times 2}]^{\GroupU} \isom \AA^3$ separates~$\GroupU$-orbits of generic points having~$a_{21} \neq 0$ or~$a_{22} \neq 0$. The image of this morphism is~$\AA^3$ without the punctured line~$\{ a_{21} = a_{22} = 0, \det \neq 0 \}$. Since the image of~$\pi$ is not open, by~\cite[Corollary~1.4]{ArzhCelik} the action~$\GroupU : \Mat_{2 \times 2}$ has no categorical quotient in the category of algebraic varieties.
\end{exa}

The morphism~$\pi$ considered in Example~\ref{simple-example} can be regarded as a quotient morphism after an appropriate modification to the definition of categorical quotient. It turns out that admitting only morphisms into \textit{varieties} as categorical quotients is overly restrictive for certain actions~$\GroupG : X$. To work around this Bialynicki-Birula introduced in~\cite{BirulaDC} the category of dense constructible subsets. This approach has been further developed in~\cite{ArzhCelik} to permit maps into constructible spaces as candidates for quotient morphisms. Recall that a constructible space is a topological space with a sheaf of functions admitting a finite cover by subsets that are isomorphic (as spaces with functions) to constructibe subsets of affine varieties. A morphism of constructible spaces is a morphism of spaces with functions. We say that a \textit{constructible quotient} is a categorical quotient in the category of constructible spaces. It is possible for an action~$\GroupG : X$ to have no quotient in the category of algebraic varieties, but to have a constructible quotient.

\begin{exa}\label{simple-example-constructible}
Let a unipotent group~$\GroupG$ act on a vector space~$V$. It follows from~\cite[Corollary~1.2]{ArzhCelik} that the action~$\GroupG : V$ admits a constructible quotient, provided that~$\kk [V]^{\GroupG}$ is finitely generated. If~$\rho : V \rightarrow \Spec \kk [V]^{\GroupG}$ is the morphism corresponding to the inclusion~$\kk[V]^{\GroupG} \subset \kk[V]$ then the constructible quotient is~$\rho : V \rightarrow \rho(V)$. In particular, the map~$\pi$ in Example~\ref{simple-example} is a constructible quotient for the action~$\GroupU : \Mat_{2 \times 2}$.
\end{exa}

Let us point out a fact concerning quotients under two commuting actions; it will be used to identify double coset varieties with quotients of homogeneous spaces. Let~$\GroupF \times \GroupH$ act on a variety~$X$ and let there be a categorical quotient~$\pi_{\GroupF} : X \rightarrow Y = X \catquot \GroupF$. The group~$\GroupH$ acts on~$Y$ as an abstract group: if~$y = \pi_{\GroupF}(x)$ then~$h \circ y = \pi_{\GroupF}(h \circ x)$. By~\cite[Theorem.~7.1.4]{Birula} this action is regular. Moreover, existence of~$Y \catquot \GroupH$ is equivalent to existence of~$X \catquot (\GroupF \times \GroupH)$ and these two quotients coincide:
\begin{equation*}
\xymatrix{
X		\ar[rr]^{\pi_{\GroupF \times \GroupH}}	\ar[rd]_{\pi_{\GroupF}}		&&		X \catquot (\GroupF \times \GroupH) = Y \catquot \GroupH.		\\
&		Y = X \catquot \GroupF		\ar[ru]_{\pi_{\GroupH}}
}
\end{equation*}

The following statement will be used in Proposition~\ref{example1-proof}.
\begin{lem}\label{uzhosnah}
Let an algebraic group~$\GroupG$ act on an algebraic variety~$Y$. Suppose that there is~$y_0 \in Y$ that belongs to closure of every~$\GroupG$-orbit. Consider the action~$\GroupG : X \times Y$, where~$X$ is a normal variety and~$\GroupG$ acts trivially on the first factor. Let~$W \subseteq X \times Y$ be a \hbox{$\GroupG$-invariant} open subset. Suppose that~$W$ contains~$X_0 \times Y$, where $X_0 \subseteq X$ is a dense subset, and~\hbox{$\pr(W) = X$}, where~$\pr$ is the projection onto the first factor. Then the action~$\GroupG : W$ has~$\pr : W \rightarrow X$ as a categorical quotient both in the category of algebraic varieties and in the category of constructible spaces.
\end{lem}
\begin{proof}
Let us fix a~$\GroupG$-invariant morphism~$\varphi : W \rightarrow Z$ into an algebraic variety~$Z$ (resp., into a constructible space) and show that it factors uniquely through~$\pr$.

Step 1. We claim that~$\varphi$ extends to a continuous map on~$W \cup X \times \{y_0\}$. Let us fix a point~$(x^\prime, y_0) \not\in X \times \{y_0\}$ and an arbitrary sequence~$\{x_n\} \subset X_0$ such that~$x_n \rightarrow x^\prime$. Now we show that the sequence~$\varphi(x_n, y_0)$ converges. Since~$\pr(W) = X$, there is a point~$(x^\prime, y)$ in~$W$ for some~$y \in Y$. The points~$(x_n, y_0)$ and~$(x_n, y)$ belong to~$W$, hence by~$\GroupG$-invariance of~$\varphi$ we have~$\varphi(x_n, y) = \varphi(x_n, y_0)$, thus~$\lim\limits_{n \rightarrow +\infty} \varphi(x_n, y_0) = \lim\limits_{n \rightarrow +\infty} \varphi(x_n, y) = \varphi(x^\prime,y)$. Since a converging sequence can have only one limit,~$\lim\limits_{n \rightarrow +\infty} \varphi(x_n, y_0)$ does not depend on choice of~$(x^\prime, y) \in W$. For the extended map~$\varphi$ we have~$\varphi(x,y) = \varphi(x,y_0)$, so continuity of~$\varphi|_{X \times \{y_0\}}$ implies continuity of~$\varphi$ on~$W \cup X \times \{y_0\}$.

Step 2. Now we show that~$X \times \{y_0\}$ can be covered by open affine sets~$X_i \times \{y_0\}$ such that the image of~$\varphi : X_i \times \{y_0\} \rightarrow Z$ is contained in some affine subset of~$Z$. Let~$\{Z_i\}$ be an affine covering of~$Z$ and $\{U_i\}$ be an affine covering of~$X$. The set~$V_{ij} = \varphi^{-1}(\varphi(U_i) \cap Z_j)$ is open in~$U_i$. Every set~$V_{ij}$ is a union of principal open subsets~$V_{ij} = \cup_k V_{ijk}$. The sets~$V_{ijk}$ make up the required covering of~$X \times \{y_0\}$.

Step 3. Since~$\varphi(x,y) = \varphi(x, y_0)$, we have~$\varphi = \tilde{\varphi} \circ \rho$, where~$\tilde{\varphi} = \varphi|_{X \times \{y_0\}}$ and~$\rho$ is the map~$W \rightarrow X \times \{y_0\}$, $\rho(x,y) = (x,y_0)$. Denote~$\imath$ the identification of~$X$ and~$X \times \{y_0\}$: $\imath(x) = (x,y_0)$. We have~$\varphi = \tilde{\varphi} \circ \imath \circ \pi$, so~$\varphi$ factors through~$\pi$. It remains to verify that~$\varphi|_{X \times \{y_0\}}$ is a morphism. The variety~$X$ is normal, hence the affine opens~$X_i \times \{y_0\}$ constructed at step~2 are normal varieties, too. Restrictions of~$\varphi$ to these opens are morphisms of affine varieties; if~$Z$ is a constructible space then~$\varphi|_{X_i \times \{y_0\}}$ is a morphism into a constructible set, but it can be regarded as a morphism into an affine variety containing~$\varphi(X_i \times \{y_0\})$. By theorem on removable singularities the continuous extensions of~$\varphi|_{X_i \times \{y_0\}}$ are morphisms.
\end{proof}

\section{Existence and non-existence of double coset varieties}\label{main-section}

\textbf{\ref{main-section}.1.} Consider a unipotent group~$\GroupG$ and a subgroup~$\GroupU$:
\begin{equation*}
\begin{matrix}
	\GroupG = \begin{pmatrix}
	1		&		*		&		*		&		*	\\
			&		1		&		*		&		*	\\
			&				&		1		&		0	\\
			&				&				&		1
	\end{pmatrix},
&
	\GroupU = \begin{pmatrix}
	1		&		*		&		0		&		0	\\
			&		1		&		0		&		0	\\
			&				&		1		&		0	\\
			&				&				&		1
\end{pmatrix},
\end{matrix}
\end{equation*}
here~$*$ denotes an arbitrary number. We claim that if we take~$\GroupF = \GroupH = \GroupU$ then the double coset variety~$\dcosets{\GroupF}{\GroupG}{\GroupH}$ does not exist. Remark that the group~$\GroupF \times \GroupH = \GroupU \times \GroupU$ is unipotent, hence every double~$(\GroupU, \GroupU)$-coset is closed~\cite[1.3]{VinbergPopov}; had~$\GroupF$ and~$\GroupH$ been reductive, this would have implied existence of the \textit{geometric} quotient~$\GroupG \rightarrow \GroupG / (\GroupF \times \GroupH)$.

\begin{stmt}
The action~$\GroupU \times \GroupU : \GroupG$ has no categorical quotient in the category of algebraic varieties. It admits a constructible quotient and the constructible quotient parametrises generic double cosets.
\end{stmt}
\begin{proof}
Consider the action of~$\GroupG$ on space of~$4 \times 2$-matrices by left multiplication. The subgroup~$\GroupU$ is the stabiliser of the matrix
\begin{equation}\label{definition-of-M}
M = \begin{pmatrix}
0		&		0		\\
0		&		0		\\
1		&		0		\\
0		&		1
\end{pmatrix}.
\end{equation}

Therefore, the homogeneous space~$\GroupG / \GroupU$ is isomorphic to~$\AA^4$ and can be identified with the variety of matrices
\begin{equation*}
\begin{pmatrix}
*		&		*		\\
*		&		*		\\
1		&		0		\\
0		&		1
\end{pmatrix}.
\end{equation*}
After this identification the action of~$\GroupU$ on~$\GroupG / \GroupU$ becomes the matrix multiplication; it is therefore isomorphic to the action of~$\GroupU$ on space~$\Mat_{2 \times 2}$ of $2 \times 2$-matrices by left multiplication. Example~\ref{simple-example} shows that this action does not admit a categorical quotient. Therefore, $\dcosets{\GroupU}{\GroupG}{\GroupU} = (\GroupG / \GroupU) \catquot \GroupU$ does not exist.

From Example~\ref{simple-example-constructible} it follows that the action~$\GroupU \times \GroupU : \GroupG$ has a constructible quotient \hbox{$\pi : \GroupG \rightarrow \dcosets{\GroupU}{\GroupG}{\GroupU} \subset \AA^3$} which separates generic double cosets.
\end{proof}

\begin{xrem}
The constructible quotient~$\pi : \GroupG \rightarrow \dcosets{\GroupU}{\GroupG}{\GroupU}$ does not separate all closed double cosets. Indeed, all~$2 \times 2$-matrices with~$a_{21} = a_{22} = 0$ (we use notation of Example~\ref{simple-example}) are fixed under the action of~$\GroupU$ and have~$\det = 0$, hence their preimages in~$\GroupG$ are closed~$(\GroupU, \GroupU)$-cosets, which are mapped by~$\pi$ to~$0 \in \AA^3$.
\end{xrem}

\textbf{\ref{main-section}.2.} Take~$\GroupG = \GL_4$ and consider the action of~$\GroupG$ on~$4 \times 2$-matrices. Let~$\GroupH$ be the stabiliser of the matrix~$M$, $M$ being the same as in~\ref{definition-of-M}. The homogeneous space~$W = \GroupG / \GroupH$ is identified with the variety of~$4 \times 2$-matrices with non-zero columns. Let~$\GroupF$ be the subgroup of~$\GroupG$ consisting of the following matrices:
\begin{equation*}
\begin{pmatrix}
1		&		a		&		0		&		0		\\
		&		1		&		0		&		0		\\
		&				&		s		&		0		\\
		&				&				&		s
\end{pmatrix},~a \in \kk,~s \in \kk^\times.
\end{equation*}
The subgroup~$\GroupF$ acts on~$W$ via matrix multiplication.

\begin{stmt}\label{example1-proof}
The action~$\GroupF \times \GroupH : \GroupG$ does not admit a categorical quotient in the category of algebraic varieties, but has a constructible quotient.
\end{stmt}
\begin{proof}
The group~$\GroupF$ is a direct product~$\GroupF = \GroupU \times \GroupS$ of one-dimensional unipotent group~$\GroupU$ and one-dimensional torus~$\GroupS$. The categorical quotient for the action~$\GroupS : W$ is~$\pr : W \rightarrow \Mat_{2 \times 2}$, where the projection~$\pr$ erases the bottom half of matrices of~$W$. Indeed, one can apply Lemma~\ref{uzhosnah} with the acting group~$\GroupS$ and $X = Y = \Mat_{2 \times 2}$ representing top and bottom halves of matrices respectively, and with~$X_0$ consisting of matrices with non-zero columns.

Thus, had the quotient~$W \catquot \GroupF = \dcosets{\GroupF}{\GL_4}{\GroupH}$ existed, it would have been also~$(W \catquot \GroupS) \catquot \GroupU = \Mat_{2 \times 2} \catquot \GroupU$, but, according to Example~\ref{simple-example}, the latter quotient does not exist.

By Lemma~\ref{uzhosnah} and Example~\ref{simple-example-constructible}, the actions~$\GroupS : W$ and~$\GroupU : \Mat_{2 \times 2}$ both have a constructible quotient, thus~$\dcosets{\GroupF}{\GL_4}{\GroupH} = (W \catquot \GroupS) \catquot \GroupU$ exists as a constructible space.
\end{proof}

\textbf{\ref{main-section}.3.} This example is based on~\cite[4.5]{ArzhCelik}. Consider the following symmetric bilinear form on~$\kk^4$:~$(e_1, e_4) = (e_2, e_3) = 1$ and the other pairings of basis vectors are zero. The cone
\begin{equation*}
X = \{ {\bf x} \in \kk^4\ |\ x_1x_4 + x_2x_3 = 0 \} \setminus \{(0,0,0,0)\}
\end{equation*}
is the collection of non-zero isotropic vectors, therefore~$X = \SO_4 / \GroupH$, where~$\GroupH$ is the stabiliser of a non-zero isotropic vector. As~$\GroupF$ we take the following unipotent subgroup of~$\SO_4$:
\begin{equation*}
\begin{pmatrix}
1		&		a		&		0		&		0		\\
		&		1		&		0		&		0		\\
		&				&		1		&		-a		\\
		&				&				&		1
\end{pmatrix},~a \in \kk.
\end{equation*}

The algebra~$\kk [X]^{\GroupF}$ is freely generated by~$x_2$ and~$x_4$; indeed, these two functions are~$\GroupF$-invariant and generic orbits meet the plane~$\{x_1 = x_3 = 0\}$, so there are no other generators. It is clear that the canonical morphism $\pi : X \rightarrow \Spec \kk [X]^{\GroupF} = \AA^2$ is surjective. Nevertheless,~$\Spec \kk [X]^{\GroupF}$ is not the quotient for~$\GroupF : X$ because the following morphism~$\varphi : X \rightarrow \PP^1$ does not factor through~$\pi$:
\begin{equation*}
\varphi(x_1, x_2, x_3, x_4) = (x_2 : x_4) = (x_1 : -x_3).
\end{equation*}
Indeed,~$\pi(x_1, x_2, x_3, x_4) = (x_2, x_4)$, and from~$\varphi = \tilde{\varphi} \circ \pi$ it follows that $\tilde{\varphi}(x_2,x_4) = (x_2 : x_4)$ when~$x_2 \neq 0$ or~$x_4 \neq 0$, hence~$\tilde{\varphi}$ is not continuous in~$(0,0)$, which is not possible.

Let us show that the considered action has a categorical quotient, though it does not coincide with~$\Spec \kk[X]^{\GroupF}$.

\begin{stmt}
The action~$\GroupF : X$ has a categorical quotient in the category of algebraic varieties.
\end{stmt}
\begin{proof}
The quotient is the blow-up of the origin in~$\AA^2$
\begin{equation*}
\hat{\AA}^2 = \{ ((x, y), (u:v)) \in \AA^2 \times \PP^1\ |\ xv - yu = 0 \}
\end{equation*}
with the morphism~$\rho : X \rightarrow \hat{\AA}^2$, \hbox{$\rho(x_1, x_2, x_3, x_4) = ((x_2, x_4), (x_1 : -x_3)) = ((x_2, x_4), (x_2 : x_4))$}.

Let us check that every~$\GroupF$-invariant morphism~$\varphi : X \rightarrow Z$ factors through~$\rho$. By~$\GroupF$-invarience of~$\varphi$ we have~$\varphi(cx_1, x_2, cx_3, x_4) = \varphi(x_1, x_2, x_3, x_4)$. Indeed, if one of~$x_2, x_4$ is not zero then the point $(x_1, x_2, x_3, x_4)$ and the points $(cx_1, x_2, cx_3, x_4)$ belong to one orbit of~$\GroupF$; if~$x_2 = x_4 = 0$ then we have
\begin{align*}
\varphi(x_1, 0, x_3, 0) =
	\lim\limits_{t \rightarrow 0} \varphi(x_1, tx_1, &x_3, -tx_3) =	\\
	&		=		\lim\limits_{t \rightarrow 0} \varphi(cx_1, tx_1, cx_3, -tx_3) =
	\varphi(cx_1, 0, cx_3, 0).
\end{align*}
Define~$\tilde{\varphi} : \hat{\AA}^2 \rightarrow Z$ as the morphism taking~$((x, y), (u:v))$ to~$\varphi(u, x, -v, y)$. The reasoning above shows that~$\tilde{\varphi}$ is well defined. Thus,~$\varphi = \tilde{\varphi} \circ \rho$, i.~e.,~$\varphi$ factors through~$\rho$. Since $\rho(X) = \hat{\AA}^2$, the morphism~$\tilde{\varphi}$ can be chosen uniquely.
\end{proof}

\begin{xrem}
It is clear that~$\rho : X \rightarrow \hat{\AA}^2$ separates orbits of points having $x_2 \neq 0$ or \hbox{$x_4 \neq 0$}. However,~$\rho$ does not separate all closed orbits: the points $z = (x_1, 0, x_3, 0)$ and~$z^\prime = (cx_1, 0, cx_3, 0)$ are~$\GroupF$-fixed, but $\rho(z) = \rho(z^\prime)$. Thus, the quotient $q : \SO_4 \rightarrow \dcosets{\GroupF}{\SO_4}{\GroupH} = \hat{\AA}^2$ separates generic double cosets, but fails to separate all closed double cosets.
\end{xrem}

\textbf{\ref{main-section}.4.} Remark that in Examples~\ref{example0} and~\ref{example1} the actions~$\GroupF \times \GroupH : \GroupG$ have no categorical quotient in the category of algebraic varieties but do admit one in the category of constructible spaces.
\begin{question}
Let~$\GroupG$ be a connected affine algebraic group and~$\GroupF, \GroupH$ be closed subgroups in~$\GroupG$. Is it true that~$\dcosets{\GroupF}{\GroupG}{\GroupH}$ exists as a constructible space?
\end{question}

The following proposition gives a partial answer to this question.

\begin{stmt}
Let~$\GroupG$ be a connected affine algebraic group and~$\GroupF, \GroupH \subset \GroupG$ be closed connected subgroups with trivial character groups. Suppose that the algebra~${}^{\GroupF} \kk[\GroupG]^{\GroupH}$ is finitely generated and let~$\pi : \GroupG \rightarrow \Spec({}^{\GroupF} \kk[\GroupG]^{\GroupH})$ be the canonical morphism. Then~$\dcosets{\GroupF}{\GroupG}{\GroupH}$ exists as a constructible space and the map \hbox{$\pi : \GroupG \rightarrow \pi(\GroupG)$} is the constructible quotient for the action of~$\GroupF \times \GroupH$ on~$\GroupG$.
\end{stmt}
\begin{proof}
By~\cite[Theorem~6]{PopovPicardGroups}, the underlying variety of~$\GroupG$ has a finite divisor class group. Additionally,~$\GroupF$ and~$\GroupH$ have trivial character groups, therefore every~$\GroupF \times \GroupH$-invariant hypersurface~$D \subset \GroupG$ is the zero set of an invariant function~$f_D \in {}^{\GroupF} \kk[\GroupG]^{\GroupH}$. It follows from~\cite[Corollary~1.2]{ArzhCelik} that the action~$\GroupF \times \GroupH : \GroupG$ has~$\pi : \GroupG \rightarrow \pi(\GroupG)$ as a constructible quotient.
\end{proof}

\begin{xrem}
One can often give a positive answer to the question on finite generation of~${}^{\GroupF} \kk[\GroupG]^{\GroupH}$. Recall that if~$\GroupR$ is a reductive group,~$Z$ is an affine $\GroupR$-variety and~$\GroupU \subset \GroupR$ is a maximal unipotent subgroup then the algebra~$\kk[Z]^{\GroupU}$ is finitely generated~\cite[Chapter~3.2]{Kraft}. Thus, the constructible space \hbox{$\dcosets{\GroupF}{\GroupG}{\GroupH} $} is guaranteed to exist if both groups~$\GroupF$ and~$\GroupH$ are maximal unipotent subgroups in bigger reductive subgroups~$\GroupF^\prime, \GroupH^\prime \subseteq \GroupG$ or if one of them is semisimple and the other one is a maximal unipotent subgroup in a bigger reductive subgroup. Other results on finite generation of algebras of invariants can be found in~\cite{Grosshans}.
\end{xrem}

\end{document}